\documentclass[11pt]{amsart}

\usepackage{amsmath, amssymb, amscd, cancel, graphicx, paralist, soul, stmaryrd}

\usepackage{mathdots}

\usepackage[all]{xy}
\usepackage{multirow}
\usepackage{longtable}
\usepackage{array}

\newtheorem{thm}{Theorem}
\newtheorem{conj}[thm]{Conjecture}

\newtheorem{lem}[thm]{Lemma}

\newtheorem{prop}[thm]{Proposition}

\def\cC{{\mathcal C}}

\def\bR{{\mathbb R}}

\def\bZ{{\mathbb Z}}

\def\del{{\partial}}

\def\int{{\text{int}}}

\begin{document}

\title[On curves intersecting at most once, II]%
{On curves intersecting at most once, II}

\author[Joshua Evan Greene]{Joshua Evan Greene}

\address{Department of Mathematics, Boston College\\ Chestnut Hill, MA 02467}

\email{joshua.greene@bc.edu}

\maketitle

\medskip

\noindent {\bf Abstract.}
We prove that on a closed, orientable surface of genus $g$, a set of simple loops with the property that no two are homotopic or intersect in more than $k$ points has cardinality $\lesssim_k g^{k+1} \log g$.
The bound matches the size of the largest known construction to within a factor of $\sim_k \log g$.
It generalizes an earlier result of the author, which treated the case $k=1$.
The proof blends probabilistic ideas with covering space arguments related to the fact that surface groups are LERF.
\vspace{.1in}

%%%
%%%
%%%

\section{Introduction.}

Let $S$ denote a connected, orientable surface of finite type and Euler characteristic $\chi < 0$, and let $k$ denote a non-negative integer.
A {\em $k$-system} on $S$ is a set of simple loops on $S$ in which no two are homotopic or intersect in more than $k$ points.
The problem at hand is to estimate the maximum cardinality of a $k$-system on $S$.
It first originated in the extremal combinatorics literature in the work of Juvan, Malnic, and Mohar \cite{jmm96}.
It independently resurfaced in the low-dimensional topology community after Farb and Leininger raised it for the case $k=1$.

Przytycki addressed the corresponding problem for simple arcs using geometric methods.
He showed that the maximum cardinality of a set of proper simple arcs on $S$ in which no two are homotopic or intersect in more than $k$ points is $\sim_k |\chi|^{k+1}$, and he determined the precise value in the case $k=1$ \cite[Theorems 1 and 5]{przytycki2015}.\footnote{For functions $f$ and $g$ of several variables including $k$, we write $f \lesssim_k g$ to mean that there exists a function $C$ of $k$ alone that satisfies the inequality $f \le C \cdot g$ for all values of the variables.
We write $f \sim_k g$ if $f \lesssim_k g$ and $g \lesssim_k f$.
Similarly, we write $f \lesssim g$ and $f \sim g$ if $C$ can be taken to be an absolute constant.
}
Conjecturally, the same estimate holds for simple loops on $S=S_g$, the closed surface of genus $g$, for which $g \sim \chi$:
\begin{conj}
\label{conj: loops}
The maximum cardinality of a $k$-system of simple loops on $S_g$ is $\sim_k g^{k+1}$.
\end{conj}

\noindent
However, Conjecture \ref{conj: loops} is not known to hold for a single value $k \ge 1$.
It is perplexing that the case of loops has proven so much more challenging than the case of arcs!
Since Przytycki's work, all progress on upper bounds for loops has reduced in some way to the case of arcs, and this paper is no different.

In the way of upper bounds, the best previous result had been that a $k$-system $\Gamma$ of simple loops has  cardinality $\lesssim_k g^{3k-1} \log g$ \cite[Theorem 4]{greene:1-system}.
The proof of that bound utilized in part a theorem of Aougab, Biringer, and Gaster that each loop $\alpha \in \Gamma$ intersects $\lesssim_k g^{3k-1}$ other loops in $\Gamma$ \cite[Theorem 1.4]{abg2017}.
We improve both bounds here for all values $k > 1$.

\begin{thm}
\label{thm: third bound}
Suppose that $\Gamma$ is a $k$-system of simple loops on $S_g$ and $\alpha \in \Gamma$.
Then the set of loops in $\Gamma$ that intersect $\alpha$ has cardinality $\lesssim_k g^{k+1}$.
\end{thm}

A familiar argument originating in the proof of \cite[Theorem 1.4]{przytycki2015} establishes the weaker statement that the set of loops in $\Gamma$ that intersect $\alpha$ {\em exactly once} has cardinality $\lesssim_k g^{k+1}$ (Lemma \ref{lem: old}).
In order to promote Lemma \ref{lem: old} to Theorem \ref{thm: third bound}, we prove the existence of a covering $\widetilde S \to S$ to which $\alpha$ lifts to a loop $\widetilde \alpha$ and to which a positive proportion of the loops in $\Gamma$ intersecting $\alpha$ all lift to loops intersecting $\widetilde \alpha$ exactly once.
Moreover, the number of sheets in the covering and the proportion of the loops are functions of $k$ alone.
By applying Lemma \ref{lem: old} to the lifted loops on $\widetilde S$, we obtain the stated bound in Theorem \ref{thm: third bound} (albeit with a larger implied function $C(k)$).
The existence of the covering involves a probabilistic idea along with a covering space argument relying on a famous theorem of Scott \cite[Theorem 3.3]{scott:1978}.

A mild refinement of Theorem \ref{thm: third bound} and the probabilistic argument in \cite{greene:1-system} then lead to:

\begin{thm}
\label{thm: second bound}
Suppose that $\Gamma$ is a $k$-system of simple loops on $S_g$.
Then $|\Gamma| \lesssim_k g^{k+1} \log g$.
\end{thm}

In the way of constructions, Aougab and, independently, Malestein, Rivin, and Theran described 1-systems of simple loops on $S_g$ of cardinality $\sim g^2$ \cite[Theorem 1.2]{aougab14}, \cite[Theorem 1]{mrt14}.
By adapting Przytycki's construction for arcs, Aougab, Biringer, and Gaster indicated the existence of $k$-systems of simple loops on $S_g$ of cardinality $\sim_k g^{k+1}$ for all even values $k$ \cite[Remark after Theorem 1.2]{abg2017}.
Drawing inspiration from this construction, we generalize it to all $k$ by hybridizing the constructions in \cite{mrt14,przytycki2015}:

\begin{prop}
\label{prop: construction}
There exists a $k$-system $\Gamma$ of simple loops on $S_g$ of cardinality $\sim_k g^{k+1}$.
Furthermore, $\Gamma$ contains a loop that intersects every other loop in $\Gamma$.
\end{prop}

Together, Theorem \ref{thm: second bound} and Proposition \ref{prop: construction} settle the opening problem for the surface $S_g$ up to a factor of $\sim_k \log g$.
We have made no effort to estimate the functions $C(k)$ implied in our bounds.
While Conjecture \ref{conj: loops} predicts that the bound in Theorem \ref{thm: second bound} is not optimal, the bound in Theorem \ref{thm: third bound} is, by comparison with Proposition \ref{prop: construction}.

In Section \ref{sec: construction}, we prove Proposition \ref{prop: construction 2}, a refinement of Proposition \ref{prop: construction} that indicates some subtlety concerning the parity of $k$.
It also suggests that the genus of a surface $S$ of given Euler characteristic may play a greater role in controlling the cardinality of a $k$-system of loops on $S$ than it does for a $k$-system of arcs on $S$; compare \cite[Theorem 1.2]{abg2017}.
In Section \ref{sec: bound}, we begin with a casual account of the thought process that led to Theorem \ref{thm: third bound}.
We then supply the detailed argument, building to a proof of Theorem \ref{thm: degree bound}, the refinement that we use to prove Theorem \ref{thm: second bound}.

%%%
%%%
%%%

\section*{Acknowledgements.}

I thank Jacob Caudell for helpful conversations.
This work was supported by NSF CAREER Award DMS-1455132.

%%%%%
%%%%%
%%%%%
%%%%%
%%%%%

\section{The construction.}
\label{sec: construction}

%%%
%%%
%%%

In this section, we state and prove a more precise version of Proposition \ref{prop: construction}:

\begin{prop}
\label{prop: construction 2}
There exists a $k$-system $\Gamma$ of simple loops on $S$ of cardinality $\sim_k |\chi|^{k+1}$ if $k$ is even and $S$ has positive genus or if $k$ is odd and $S$ has genus $\ge |\chi| / k$.
Furthermore, in both cases, $\Gamma$ contains a loop that intersects $\sim_k |\chi|^{k+1}$ other loops in $\Gamma$.
\end{prop}

The case in which $k$ is even was already indicated in \cite[Remark after Theorem 1.2]{abg2017}.
We supply the implied construction here, as it inspires the case in which $k$ is odd, which is new.
The case in which $k$ is odd hybridizes the constructions of $k$-systems of simple arcs due to Przytycki and 1-systems of simple loops due to Malestein, Rivin, and Theran.

The restriction on the genus in the case that $k$ is odd seems subtle.
For instance, it is necessary when $k = 1$: it follows from \cite[Theorem 1.2]{mrt14} that a 1-system of simple loops on a surface of Euler characteristic $\chi$ and genus $g$ has cardinality $\lesssim_g |\chi|$.
What about the case $k=3$?  
What is the maximum cardinality of a 3-system of simple loops on a surface of genus 1 and Euler characteristic $\chi$?
We only know that it is bounded between $\sim |\chi|^3$, coming from the 2-system in Proposition \ref{prop: construction 2}, and $\sim |\chi|^4 \log |\chi|$, coming from Theorem \ref{thm: second bound}.

\begin{proof}
{\em Case one: $k$ is even.}
Przytycki constructed a $k$-system $A$ of proper, simple arcs on a compact planar surface $P$ of Euler characteristic $\chi$ \cite[Theorem 1.5]{przytycki2015}.
In this construction, $|A| \sim_k |\chi|^k$; there exists an arc $a \in A$ that is disjoint from every other arc in $A$; and $P$ has two distinguished boundary components with the property that each arc in $A$ has one endpoint on each component.
Attach an annulus to $P$ by gluing its boundary components to the distinguished boundary components of $P$ so as to produce a surface $T$ of genus one and Euler characteristic $\chi$.
Join the endpoints of the arc $a$ by an arc $b$ in the annulus, and do the same for every other arc in $A$, subject to the condition that the added arcs are disjoint from $b$.
The result is a set $\Gamma$ of pairwise non-homotopic simple loops on $T$.
It has cardinality $|\Gamma| = |A| \sim_k |\chi|^{k+1}$, and it contains a non-separating simple loop $\gamma = a \cup b$ that is disjoint from every other loop in $\Gamma$.

Consider a pair of loops in $\Gamma$ distinct from $\gamma$.
They are contained in the planar subsurface $T - \gamma$, so they have even intersection number.
On the other hand, they intersect in at most one more point than the corresponding arcs in $A$, hence in at most $k+1$ points.
Thus, if $k$ is even, then any two loops in $\Gamma$ intersect in at most $k$ points.
It follows that $\Gamma$ is a $k$-system of simple loops on $T$ of the desired cardinality.
In order to produce such a $k$-system on an arbitrary surface of Euler characteristic $\chi$ and positive genus, simply attach the necessary number of annuli to $T$.

\noindent
{\em Case two: $k$ is odd.}
Select positive integers $g$ and $h$.
Begin with a regular $(4g+2)$-gon and select a main diagonal of it.
Take $k$ chords of the polygon that are parallel and close to the diagonal.
They cut the polygon into $k-1$ strips and two other polygonal regions.
Remove $h+1$ small open disks along each chord.
Consider an arc that connects the midpoints of opposite sides of the polygon and that crosses each of the $k$ chords exactly once, passing between a pair of the holes on each.
The number of such arcs, up to homotopy rel endpoints, is $(2g+1) \cdot h^k$.
Now glue the opposite sides of the polygon in pairs so as to obtain a surface $S$ of genus $g$ with $k(h+1)$ boundary components.
The arcs close up to $(2g+1) \cdot h^k$ pairwise non-homotopic simple loops on $S$.

We claim that these loops constitute a $k$-system.
First, consider a pair of loops corresponding to arcs joining midpoints of different pairs of opposite sides.
The intersection number between these loops is odd.
One argument for why is by filling the boundary components: doing so does not change the$\pmod 2$ intersection pairing between the loops, and the loops on the filled surface can be isotoped to meet at one transverse point of intersection in the center of the polygon.
In addition, the loops can be homotoped so that there is at most one intersection point between them in each of the $k-1$ strips and in each of the two polygonal regions.
Hence their intersection number is both at most $k+1$ and odd, so in fact it is at most $k$.
Second, consider a pair of loops corresponding to arcs joining midpoints of the same pair of opposite sides.
They can be homotoped so that there is at most one intersection point between them in each of the $k-1$ strips and at most one more point, namely at the midpoint of the identified edge pair.
Hence their intersection number is at most $k$.
(In fact, it is even, so it is at most $k-1$.)

In total, we have constructed a $k$-system of simple loops of cardinality $(2g+1) \cdot h^k$ on a surface of genus $g$ and Euler characteristic $2-2g-k(h+1)$.
Taking $g = h$ gives a $k$-system of cardinality $\sim g^{k+1}$ on a surface of genus $g$ and Euler characteristic $\sim gk$.
Thus, we obtain a $k$-system of cardinality $\sim_k |\chi|^{k+1}$ on a surface of genus $|\chi|/k$ and Euler characteristic $\chi$.
Attaching annuli as before realizes all possible genera $\ge |\chi| / k$.

Finally, we augment the $k$-system by including the core of the first attached annulus in the case $k$ is even and the main diagonal in the case $k$ is odd.
The result is a $k$-system of the required cardinality that contains a loop that intersects every other loop in the $k$-system.
\end{proof}

In fact, it is not hard to see that there exists a function $p: \bZ^+ \to \bR^+$ with the property that a proportion $p(k)$ of the loops in $\Gamma$ have the property that they intersect with a proportion of at least $p(k)$ other loops in $\Gamma$.
Thus, many loops in $\Gamma$ meet $\sim |\chi|^{k+1}$ others.

%%%
%%%
%%%

\section{The upper bound.}
\label{sec: bound}

\subsection{Overview.}
We begin by describing the thought process that led to Theorem 2.

Suppose that $\Gamma$ is a 2-system on $S$, fix a loop $\alpha \in \Gamma$, and let $d$ denote the number of loops in $\Gamma$ that intersect $\alpha$.
We would like to argue that $d \lesssim |\chi|^3$.
To do so, we sort the $d$ loops into three subsets: $d_1$ loops that meet $\alpha$ in a single point, $d_2$ loops that meet $\alpha$ in two points of the same sign, and $d_3$ loops that meet $\alpha$ in two points of opposite signs.
We would like to argue that $d_i \lesssim |\chi|^3$ for $i=1,2,3$ in turn.

The first key idea is that we can bound $d_1$ by a familiar argument:

\begin{lem}
\label{lem: old}
Suppose that $\Gamma$ is a $k$-system of simple loops on $S$ and $\alpha \in \Gamma$.
The number of loops in $\Gamma$ that intersect $\alpha$ exactly once is $\lesssim_k |\chi|^{k+1}$.
\end{lem}

\noindent
We alluded to this statement in the Introduction, and we establish a more general result in Lemma \ref{lem: sort of old} below.
Thus, $d_1 \lesssim |\chi|^3$, by Lemma \ref{lem: old}.

The second key idea is that we can bound $d_2$ in a similar way, once we pass to an appropriate covering of $S$.
Specifically, form the double cover $\widetilde S \to S$ obtained by cutting $S$ along $\alpha$, taking two copies of the result, and gluing them end on end.
Intrinsically, $\widetilde S$ is the double cover corresponding to the mapping $\pi_1(S) \to H_1(S) \to \bZ / 2 \bZ$, where the second map is$\pmod 2$ intersection pairing with $[\alpha]$.
Under this covering, $\alpha$ lifts to a loop $\widetilde \alpha$, and each of the $d_2$ loops $\beta$ that intersects $\alpha$ in two points of the same sign lifts to a loop $\widetilde \beta$ that intersects $\widetilde \alpha$ in a single point.
Moreover, the lifts of any two such loops meet in at most two points, and no two are homotopic.
Hence we have a 2-system of $d_2$ loops on $\widetilde S$ that each meet $\widetilde \alpha$ in a single point.
Thus, $d_2 \lesssim |\chi(\widetilde S)|^3 \sim |\chi|^3$ by Lemma \ref{lem: old}, noting that $\chi(\widetilde S) = 2 \chi$: the complexity of $\widetilde S$ is bounded by that of $S$.

The third key idea is that we can attempt to bound $d_3$ by passing to a randomly selected covering.
Suppose that $\alpha$ intersects $\beta$  in two points (of either sign), and let $\gamma$ denote a simple loop formed from an arc of $\alpha$ and an arc of $\beta$  between their intersection points.
A non-separating loop lifts to approximately half of the double covers of $S$, which are parametrized by $H^1(S;\bZ/2\bZ)$.
Thus, if we fix a non-separating loop and select a double cover of $S$ uniformly at random, then it has probability about 1/2 of lifting.
Assuming that $\alpha$, $\beta$, and $\gamma$ are non-separating and the events of each lifting to a random double cover are independent events, then with probability about $1/8$, $\alpha$ and $\beta$  will lift, but $\gamma$ will not.
In such a cover, $\alpha$ and $\beta$  have lifts intersecting in a single point.
Thus, by averaging, in a randomly selected double cover to which $\alpha$ lifts to a loop $\widetilde \alpha$, we expect about 1/4 of the $d_2 + d_3$ loops $\beta \in \Gamma$ that intersect $\alpha$ twice to lift a loop that meets $\widetilde \alpha$ once.
Hence there exists at least one double cover $\widetilde S \to S$ in which $\alpha$ lifts to a loop $\widetilde \alpha$ and at least about 1/4 of these $d_2 + d_3$ loops lift to loops hitting $\widetilde \alpha$ once.
We thereby obtain $(1/4)(d_2+d_3) \lesssim |\chi(\widetilde S)|^3$ by applying Lemma \ref{lem: old} to the 2-system of $\sim (1/4)(d_2+d_3)$ lifted loops in $\widetilde S$.
Once more, this leads to the desired bound $d_2 + d_3 \lesssim |\chi|^3$.

However, the attempted argument makes some overly optimistic suppositions: the events of lifting may not be independent, and loops may be separating.
In fact, in the case that $\Gamma$ consists solely of separating loops, the loops of $\Gamma$ lift to every double cover, and no two lifts will intersect in just one point.
Thus, the strategy as articulated breaks down entirely in this case.

The fourth and final key idea is to overcome these obstacles by passing to a more complicated set of coverings guaranteed by a famous theorem of Scott \cite[Theorem 3.3]{scott:1978}.
We arrived at it in the pursuit of coverings to which $\alpha$ and $\beta$ lift but $\gamma$ does not, a special case of the fact that surface groups are LERF.
For a pair of simple loops on $S$ that meet in two or fewer points, there exists an incompressible subsurface $S_0 \subset S$ that contains the two and whose complexity is independent of $S$.
Scott's theorem leads to the existence of a covering $\widetilde S_0 \to S_0$, also of bounded complexity, to which the two loops lift to loops that intersect exactly once (Lemma \ref{lem: subsurface}).
By carefully constructing coverings $\widetilde S \to S$ and counting the ones that extend $\widetilde S_0 \to S_0$, we obtain absolute constants $f, r > 0$ and a set $\cC(S)$ of coverings $\widetilde S \to S$, each with $\le f$ sheets, with the property that any pair of loops on $S$ that meet in at most two points lift to a pair of loops intersecting exactly once in a proportion of $r$ of the coverings in $\cC(S)$ (Theorem \ref{thm: lifting}).
Given $\Gamma$ and $\alpha$, averaging shows that there exists a covering $\widetilde S \to S$ in $\cC(S)$ to which $\alpha$ lifts to a loop $\widetilde \alpha$ and at least $r d / f$ loops in $\Gamma$ lift to loops intersecting $\widetilde \alpha$ exactly once (Theorem \ref{thm: degree bound}).
As before, these loops form a $2$-system on $\widetilde S$ to which we can apply Lemma \ref{lem: old}, leading to the bound $r  d / f \lesssim |\chi(\widetilde S)|^3 \le f^3 |\chi|^3$.
Hence $d \lesssim |\chi|^3$, as desired.

%%%
%%%
%%%

\subsection{The proofs.}
The strategy just described applies directly to $k$-systems of simple loops for all $k \ge 1$.  
We now supply the detailed arguments.
For a pair of simple loops $\alpha$ and $\beta$ on a surface $S$, let $\iota(\alpha,\beta)$ denote their geometric intersection number: this is the fewest number of points in which a loop homotopic to $\alpha$ intersects a loop homotopic to $\beta$.

\begin{lem}
\label{lem: subsurface}
Suppose that $\alpha$ and $\beta$ are simple loops on $S$ in minimal position and $1 \le \iota(\alpha,\beta) \le k$.
For any basing of $\alpha$ and $\beta$ at a point of $\alpha \cap \beta$, there exists a compact, connected, $\pi_1$-injective subsurface $S_0 \subset S$ and a covering $p_0: \widetilde S_0 \to S_0$ such that
\begin{enumerate}
\item
$S_0$ contains $\alpha$ and $\beta$;
\item
each component of $\del S_0$ is either a component of $\del S$ or separates a subsurface of $S$; 
\item
$b_1(S_0) \le 2k+2$ and $t:= |\pi_0(\del S_0)| \le k+2$;
\item
$\alpha$ and $\beta$ lift to $\widetilde S_0$ to based loops intersecting exactly once; and
\item
the number of sheets in the covering is at most $g(k)$, for some function $g: \bZ^+ \to \bZ^+$.
\end{enumerate}
\end{lem}

Properties (1)-(3) are straightforward to achieve.
The crux of the matter is the theorem of Scott that guarantees property (4).
Property (5) issues easily from (1)-(4).

\begin{proof}
First, form a subsurface $S' \subset S$ by taking a closed, regular neighborhood of $\alpha \cup \beta$ along with every planar component of $S - \alpha \cup \beta$.
Let $S'_1,\dots,S'_t$ denote the closures of the components of $S - S'$.
For each $j=1,\dots,t$, we can locate a planar subsurface $P_j \subset S'_j$ such that $\del P_j$ consists of $\del S' \cap \del S'_j$ and a single loop interior to $S'_j$.
We let $S_0 = S' \cup P_1 \cup \cdots \cup P_t$, and we let $S_j$ denote the closure of $S_j' - P_j$ for $j=1,\dots,t$.
Properties (1) and (2) above are now immediate, and (3) is straightforward to establish.
In particular, the number of homeomorphism types the triple $(S_0,\alpha,\beta)$ may assume is a finite value $h(k)$.

Place a basepoint on $S_0$ at a point of $\alpha \cap \beta$, and let $H \le \pi_1(S_0)$ denote the subgroup generated by the classes of $\alpha$ and $\beta$.
By \cite[Theorem 3.3]{scott:1978}, there exists a finite-sheeted covering $p_0: \widetilde S_0 \to S_0$ and a $\pi_1$-injective subsurface $F \subset \widetilde S_0$ such that $H = (p_0)_* \, \pi_1(F)$.
Since $\alpha$ and $\beta$ are non-isotopic simple loops, $H$ has rank two, and $b_1(F)=2$.
Let $\widetilde \alpha$ and $\widetilde \beta$ denote the lifts of $\alpha$ and $\beta$ to $F$.
They are simple based loops that generate $\pi_1(F)$, so $\iota(\widetilde \alpha, \widetilde \beta) = 1$ or 0 according to whether $F$ has positive genus or not.
Since $\alpha$ and $\beta$ are in minimal position, the same is true of $\widetilde \alpha$ and $\widetilde \beta$; and since $\alpha$ and $\beta$ meet transversely at the basepoint of $S_0$, the lifts $\widetilde \alpha$ and $\widetilde \beta$ meet transversely at the lift of the basepoint to $F$.
It follows that $\widetilde \alpha$ and $\widetilde \beta$ intersect exactly once.
This establishes property (4).

Finally, the choice of covering $p_0: \widetilde S_0 \to S_0$ depends only on the homeomorphism type of the triple $(S_0,\alpha,\beta)$ and not the embedding $S_0 \subset S$.
Since the number of these homeomorphism types is a finite value $h(k)$, the number of sheets in the covering is bounded as well by a finite value $g(k)$.
This establishes (5) and completes the proof.
\end{proof}

The next result promotes the covering produced in Lemma \ref{lem: subsurface} to many coverings of $S$ by parameter counting.
We emphasize our use of terminology: in the absence of a basepoint, we say that a loop $\alpha \subset S$ lifts to a loop $\widetilde \alpha \subset S$ under a covering $p: \widetilde S \to S$ if $\widetilde \alpha$ is a component of $p^{-1}(\alpha)$ and $p$ is one-to-one on $\widetilde \alpha$.

\begin{thm}
\label{thm: lifting}
There exist functions $f: \bZ^+ \to \bZ^+$ and $r : \bZ^+ \to \bR^+$
and a finite set $\cC(S)$ of connected coverings of $S$ with the following properties.
Suppose that $\alpha$ and $\beta$ are loops on $S$ in minimal position and $1 \le \iota(\alpha,\beta) \le k$.
Then $\alpha$ and $\beta$ lift to loops intersecting exactly once
in a proportion of at least $r(k)$ of the coverings in $\cC(S)$.
Furthermore, each covering in $\cC(S)$ has at most $f(k)$ sheets.
\end{thm}

We have made no effort to control the functions $f$ and $r$.

\begin{proof}
First, select a subsurface $S_0 \subset S$, place a basepoint at a point of $\alpha \cap \beta$, and construct a covering $p_0: \widetilde S_0 \to S_0$ as in Lemma \ref{lem: subsurface}.

\noindent
{\em From coverings to subgroups.}
We translate the covering $p_0 : \widetilde S_0 \to S_0$ into the data of a map $\phi : \pi_1(S_0) \to G_k$ and a subgroup $J \le G_k$, where $G_k$ denotes a finite group that depends only on $k$.
Corresponding to $p_0: \widetilde S_0 \to S_0$ is the subgroup $H = (p_0)_* \pi_1(\widetilde S_0) \le \pi_1(S_0)$ of index $\le g = g(k)$.
The intersection of the finitely many conjugates of $H$ in $\pi_1(S_0)$ is a normal subgroup $N$ of finite index.
In fact, its index is at most $g!$, by a theorem of Poincar\'e.
Thus, the quotient group $\pi_1(S_0)/N$ is a finite group of order at most $g!$.
Under the quotient map $\pi_1(S_0) \to \pi_1(S_0)/N$, $H$ is the preimage of the subgroup $H/N$.
We simultaneously embed all groups of order $\le g!$ into the symmetric group $K$ on $g!$ letters.
We then embed $K$ as the diagonal subgroup $K \times K$ in the alternating group on $2 g!$ letters.
This is the group $G_k$; the reason we take it to be an alternating group is explained when we construct extensions below.
Composing the quotient map with the embedding yields a map $\phi: \pi_1(S_0) \to G_k$ and a subgroup $J \le G_k$ with the property that $H = \phi^{-1}(J)$.

\noindent
{\em Promoting coverings.}
If we extend $\phi_0 : \pi_1(S_0) \to G_k$ to a map $\phi: \pi_1(S) \to G_k$, then the preimage of $J$ under $\phi$ intersects the subgroup $\pi_1(S_0) \le \pi_1(S)$ precisely in $H$.
Therefore, the covering $p: \widetilde S \to S$ corresponding to the subgroup $\phi^{-1} (J)$ restricts to a (possibly disconnected) covering of $S_0$, and the component of this covering that contains the lift of the basepoint is the covering $p_0: \widetilde S_0 \to S_0$.
In particular, $\alpha$ and $\beta$ lift to simple loops intersecting exactly once on $\widetilde S$.

\noindent
{\em Constructing extensions.}
We now take up the problem of constructing such extensions.
We use the notation introduced in the proof of Lemma \ref{lem: subsurface}.
Fix a component $S_j$ of the closure of $S - S_0$.
By Lemma \ref{lem: subsurface} part (2), it intersects $S_0$ in a single simple loop that is a boundary component of each.
Let $c$ denote a based loop in $S_0$ freely homotopic to this loop.
Let $h$ denote the genus and $m$ the number of boundary components of $S_j$, and place a basepoint on $S_j$.
We can find based loops $a_1,b_1,\dots,a_h,b_h,c_1,\dots,c_m$ in $S_j$ such that
\begin{itemize}
\item
$a_1,b_1,\dots,a_h,b_h$ are freely homotopic to a geometric symplectic basis of $S_j$;
\item
$c_1,\dots,c_m$ are freely homotopic to the boundary components and punctures of $S_j$;
\item
$c_m$ is freely homotopic to the same boundary component as $c$; and
\item
$\pi_1(S_j) = \langle a_1,b_1,\dots,a_h,b_h,c_1,\dots,c_m \, | \, \prod_{i=1}^h [a_i,b_i] \cdot \prod_{j=1}^{m-1} c_j = c_m \rangle$.
\end{itemize}
By the Seifert - van Kampen theorem, $\pi_1(S_0 \cup S_j)$ is isomorphic to the free product of $\pi_1(S_0)$ and $\pi_1(S_j)$ subject to the single relation $c = c_m$.

Thus, in order to extend $\phi_0$ from $\pi_1(S_0)$ to $\pi_1(S_0 \cup S_j)$, we must specify a value $\phi (g) \in G_k$ for each generator $g \in \{ a_1,b_1,\dots,a_h,b_h,c_1,\dots,c_{m-1} \}$ subject to the single relation
\begin{equation}
\label{eq: relation}
\prod_{i=1}^h [\phi(a_i),\phi(b_i)] \cdot \prod_{j=1}^{m-1} \phi (c_j) = \phi_0(c).
\end{equation}
Observe that if $m > 1$, then we can specify arbitrary values for each $g \ne c_{m-1}$ and then fix the value of $\phi(c_{m-1})$ to ensure that \eqref{eq: relation} holds.
If instead $m=1$, then $h \ge 1$, since $S_0$ is $\pi_1$-injective, and we can still specify arbitrary values for each $g \ne a_h, b_h$.
Then we must select $\phi(a_h)$ and $\phi(b_h)$ so that their commutator equals $( \prod_{i=1}^{h-1} [\phi(a_i),\phi(b_i)] )^{-1} \cdot \phi_0(c)$: this is possible to do, since every element of the alternating group $G_k$ is a commutator, by \cite[Theorem 1]{miller:1899} (assuming, by a slight enlargement if need be, that $G_k$ is an alternating group on $n \ge 5$ letters).
We carry out this procedure for each component $S_j$, $j=1,\dots,t$, and in this way we obtain an extension $\phi : \pi_1(S) \to G_k$.

\noindent
{\em Counting extensions.}
Now we can estimate how many such extensions there are of $\phi_0$.
An extension is determined by its values on $b_1(S)$ generators of $\pi_1(S)$, which we may take to be a subset of the $b_1(S_0) \le 2k+2$ generators of $\pi_1(S_0)$ and the $b_1(S_j)$ generators of $\pi_1(S_j)$ for each $j = 1,\dots,t \le k+2$.
The initial map $\phi_0$ is determined by its values on the generators of $\pi_1(S_0)$, and the extension is freely specified on all but at most two generators of each $\pi_1(S_j)$.
Hence the number of constrained parameters is at most $(2k+2) + (k+2) = 3k+4$.
The number of extensions of $\phi_0$ to $\phi: \pi_1(S) \to G_k$ is consequently at least $f(k)^{b_1(S) - 3k - 4}$, where $f(k) := |G_k|$.

\noindent
{\em The set of coverings.}
We let $\cC(S)$ denote the set of coverings of $S$ which correspond to the preimage of a subgroup $J \le G_k$ under a homomorphism $\phi: \pi_1(S) \to G_k$.
Note that $\cC(S)$ is independent of the choice of basepoint on $S$, and each covering in $\cC(S)$ contains at most $f(k)$ sheets.
The cardinality of $\cC(S)$ is bounded above by the number of homomorphisms $\phi: \pi_1(S) \to G_k$ times the number of subgroups of $G_k$.
By parameter counting, the number of homomorphisms is bounded above by $f(k)^{b_1(S)}$, while the number of subgroups is a function $s(k)$ of $k$ alone.
Hence $|\cC(S)| \le f(k)^{b_1(S)} s(k)$.

\noindent
{\em Final\'e.}
We have shown that $\ge f(k)^{b_1(S) - 3k - 4}$ coverings in $\cC(S)$ have the property that $\alpha$ and $\beta$ lift to loops intersecting exactly once.
Hence the proportion of coverings in $\cC(S)$ with this property is at least $r(k) := f(k)^{-3k-4} s(k)^{-1}$.

\end{proof}

Given an arbitrary subset of loops $\Gamma_0 \subset \Gamma$, we define two more subsets of $\Gamma$ related to $\Gamma_0$:
\begin{itemize}
\item
$A(\Gamma_0)$, the subset of loops in $\Gamma$ with a single intersection point with the loops in $\Gamma_0$;
\item
$U(\Gamma_0)$, the subset of loops in $\Gamma$ that intersect a unique loop in $\Gamma_0$.
\end{itemize}
Thus, $A(\Gamma_0) \subset U(\Gamma_0)$.
Note that if $\Gamma$ is a 1-system, then $A(\Gamma_0) = U(\Gamma_0)$ for any $\Gamma_0 \subset \Gamma$.
This observation serves to explain why the distinction between these two subsets did not arise in \cite{greene:1-system}.

The next result is the promised refinement of Lemma \ref{lem: old}.
It does not involve any of the preparation developed so far.
Rather, its proof adapts a familiar argument that first appeared in the proof of \cite[Theorem 1.4]{przytycki2015}.

\begin{lem}
\label{lem: sort of old}
Suppose that $\Gamma$ is a $k$-system on $S$ and $\Gamma_0 \subset \Gamma$.
Then $|A(\Gamma_0)| \lesssim_k |\chi|^{k+1}$.
\end{lem}

\begin{proof}
The union of the loops in $\Gamma_0$ is a 1-complex in $S$.
Choose a component of it and resolve its intersection points in such a way that the result is connected.
Let $\Gamma_1$ denote the union of the resolved components that intersect a loop in $A(\Gamma_0)$.
Thus, $\Gamma_1$ is a collection of pairwise disjoint, simple loops on $S$; moreover, for each loop in $\Gamma_1$, there is a loop in $A(\Gamma_0)$ that meets it in a single point and which is otherwise disjoint from $\Gamma_1$.
Therefore, cutting $S$ along $\Gamma_1$ results in a connected surface with the same Euler characteristic as $S$.
Moreover, each loop of $A(\Gamma_0)$ cuts open to an arc, and any two of the arcs intersect at most $k$ times.
Any two loops in $A(\Gamma_0)$ cutting open to homotopic arcs on the cut surface must intersect the same component $\alpha$ of $\Gamma_1$ and differ by a power of Dehn twist along $\alpha$.
This power is at most $k$, since the loops intersect in at most $k$ points.
Therefore, at most $k+1$ loops in $A(\Gamma_0)$ can cut open to the same homotopy type of arc on the cut surface.
Consequently, the number of loops in $A(\Gamma_0)$ is at most $k+1$ times the size of a $k$-system of arcs on the cut surface, which is $\lesssim_k |\chi|^{k+1}$ by \cite[Theorem 1.5]{przytycki2015}.
\end{proof}

Next, we apply Theorem \ref{thm: lifting} in order to promote Lemma \ref{lem: sort of old} to a much stronger result:

\begin{thm}
\label{thm: degree bound}
Suppose that $\Gamma$ is a $k$-system on $S$ and $\Gamma_0 \subset \Gamma$.
Then $|U(\Gamma_0)| \lesssim_k |\chi|^{k+1}$.
\end{thm}

\noindent
This is the stated refinement of Theorem \ref{thm: third bound}, which corresponds to the case in which $\Gamma_0$ consists of a single loop.
This case is really at the heart of the following proof.
The general case presented requires just a little extra dressing.

\begin{proof}
As in the proof of Lemma \ref{lem: sort of old}, the union of the loops in $\Gamma_0$ is a 1-complex in $S$, and we resolve its intersection points to produce one simple loop for each of its components.
Let $\Gamma_1$ denote the union of the resolved components that essentially intersect some loop in $U(\Gamma_0)$.
Thus, $\Gamma_1$ is a collection of pairwise disjoint, simple loops on $S$ with the property that each loop in $U(\Gamma_0)$ intersects a unique loop in $\Gamma_1$, and it does so in $\le k$ points.

For each loop $\beta \in U(\Gamma_0)$, there exists a unique loop $\alpha \in \Gamma_1$ that it essentially intersects.
Let $P$ denote the set of such pairs $(\alpha,\beta) \in \Gamma_1 \times U(\Gamma_0)$.
Apply Theorem \ref{thm: lifting}.
For each pair $(\alpha,\beta) \in P$, a proportion of at least $r(k)$ of the coverings in $\cC(S)$ have the property that $\alpha$ and $\beta$ lift to simple loops $\widetilde \alpha$, $\widetilde \beta$ with $\iota(\widetilde \alpha, \widetilde \beta) = 1$.
Therefore, by averaging, there exists a covering $p: \widetilde S \to S$ in $\cC(S)$ with $\le f(k)$ sheets and a subset $P_0 \subset P$ of cardinality $|P_0| \ge r(k) \cdot |P| = r(k) \cdot |U(\Gamma_0)|$ with the property that for all $(\alpha,\beta) \in P_0$, $\alpha$ and $\beta$ lift to loops $\widetilde \alpha$, $\widetilde \beta$ with $\iota(\widetilde \alpha, \widetilde \beta) = 1$.

Note that a given loop $\alpha$ may occur in several pairs in $P_0$, and the associated lifts of it may be different.
However, $p^{-1}(\alpha)$ contains at most $f(k)$ components, so for a proportion of at least $1/f(k)$ of the pairs in $P_0$ with first coordinate $\alpha$, the lifts $\widetilde \alpha$ are the same.
Thus, we may pass to a subset $P_1 \subset P_0$ of cardinality $|P_1| \ge |P_0| / f(k)$ with the additional property that for any two pairs $(\alpha,\beta_1)$, $(\alpha,\beta_2) \in P_1$, the associated lifts of $\alpha$ are the same.

Let $\widetilde \Gamma_1$ denote the set of such lifts $\widetilde \alpha$ and $\widetilde U(\Gamma_0)$ the set of such lifts $\widetilde \beta$, where $(\alpha,\beta) \in P_1$.
We have $|\widetilde U(\Gamma_0)| = |P_1|$.
No two of the loops in $\widetilde \Gamma_1 \cup \widetilde U(\Gamma_0)$ are homotopic or intersect in more than $k$ points.
Therefore, this set is a $k$-system of simple loops on $\widetilde S$.
The loops of $\widetilde \Gamma_1$ are pairwise disjoint, and each loop in $\widetilde U(\Gamma_0)$ intersects a single loop in $\widetilde \Gamma_1$, which it does so in a unique point.
Thus, $\widetilde U (\Gamma_0) = A (\widetilde \Gamma_1)$, and Lemma \ref{lem: sort of old} yields $|A(\widetilde \Gamma_1)| \lesssim_k |\chi(\widetilde S)|^{k+1}$.
Since $r(k) \cdot |U(\Gamma_0)| / f(k) \le | A(\widetilde \Gamma_1)|$ and $|\chi(\widetilde S)| \le f(k) \cdot |\chi|$, it follows that $|U(\Gamma_0)| \lesssim_k |\chi|^{k+1}$.
\end{proof}

At last, we establish our main result.

\begin{proof}[Proof of Theorem \ref{thm: second bound}]
If $\Gamma$ contains a loop disjoint from the rest, then the result follows by an easy induction on the genus.
If not, then we apply \cite[Theorem 3]{greene:1-system} to find a subset $\Gamma_0 \subset \Gamma$ with the property that $|U(\Gamma_0)| \gtrsim_k |\Gamma| / \log g$ and then apply Theorem \ref{thm: degree bound} to it.
\end{proof}

\bibliographystyle{myalpha}
\bibliography{/Users/JoshuaGreene/Dropbox/Papers/References}
%\bibliography{/Users/greenegh/Dropbox/Papers/References}
%\bibliography{Dropbox/Papers/References}

\end{document}